\documentclass[preprint]{elsarticle}

\usepackage{amsfonts,amssymb,amsmath}

\newenvironment{proof}{\medskip                    
\noindent{\scshape Proof:}}{\quad $\square$
\medskip}  

\usepackage{hyperref}

\newtheorem{theorem}{Theorem}[section]
\newtheorem{lemma}[theorem]{Lemma}

\newtheorem{corollary}[theorem]{Corollary}
\newtheorem{example}[theorem]{Example}
\newtheorem{remark}[theorem]{Remark}
\newtheorem{algorithm}[theorem]{Algorithm}

\newcommand{\RR}{\mathbb{R}}                 

\def\cT{{\mathcal T}}

\def\indeg{\operatorname{indeg}}
\def\inset{\operatorname{in}}
\def\semr{S}
\def\toT{\overset{T}{\to}}

\def\msn{\smallskip\noindent}

\def\sssn{\noindent}

\begin{document}

\title{The Markov Chain Tree Theorem and the State 
Reduction Algorithm in Commutative Semirings}

\author[rvt1]{Buket Benek Gursoy}
\ead{buket.gursoy@ichec.ie}

\author[rvt2]{Steve Kirkland\fnref{fnsk}}
\ead{steven.kirkland@umanitoba.ca}

\author[rvt3]{Oliver Mason\fnref{fnom}}
\ead{oliver.mason@nuim.ie}

\author[rvt4]{Serge{\u\i} Sergeev\fnref{fns}\corref{cor}}
\ead{sergiej@gmail.com}

\address[rvt1]{Irish Centre for High-End Computing
(ICHEC), Grand Canal Quay, Dublin 2, Ireland }
\address[rvt2]{Department of Mathematics, University of Manitoba, Winnipeg, R3T 2N2 
Canada}
\address[rvt3]{Hamilton Institute, National University of Ireland 
Maynooth, Maynooth, Co. Kildare, Ireland}
\address[rvt4]{University of Birmingham, School of Mathematics, 
Edgbaston B15 2TT, UK}

\cortext[cor]{Corresponding author. Email: sergiej@gmail.com}

\fntext[fnom]{Supported by the Irish 
Higher Educational Authority (HEA) PRTLI Network Mathematics Grant.}
\fntext[fnsk]{Research supported in part by the University of 
Manitoba under grant 315729-352500-2000.}
\fntext[fns]{Supported by EPSRC grant EP/J00829X/1 and RFBR grant 12-01-00886}

\begin{abstract}
We extend the Markov chain tree theorem to general commutative semirings, and 
we generalize the state reduction algorithm to general commutative semifields. 
This leads to
a new universal algorithm, whose prototype is the state reduction algorithm which 
computes the Markov chain tree vector of a stochastic matrix.  
\end{abstract}


\maketitle

\section{Introduction}
\label{s:intro}

The Markov Chain Tree Theorem 
states 
that each (row) stochastic matrix $A$ has a left eigenvector $x$, 
such that each entry $x_i$ is the sum of the weights of all spanning 
trees rooted at $i$ and with edges directed towards $i$. 
This vector has all components positive if $A$ is irreducible, 
and it can be $0$ in the general case.
It can be computed by means of the State Reduction Algorithm 
formulated independently by Sheskin~\cite{She}
and Grassman, Taksar and Heyman~\cite{GTH}; see also Sonin~\cite{Son99} 
for more information on this.

In the present paper, our main goal is to generalize this algorithm to matrices over commutative semifields, inspired by the
ideas of Litvinov et al.~\cite{LM98,LMa,LRSS}. To this end, let us mention first the 
tropical mathematics~\cite{BCOQ92,But:10,HOW}, which is 
a relatively new branch of mathematics developed 
over idempotent semirings, of which 
the tropical semifield, also known as the max algebra, 
is the most useful example. In one of its equivalent 
realizations (see Bapat~\cite{Bap98}), the max algebra is just the set of nonnegative real numbers
equipped with the two operations $a\oplus b=\max(a,b)$ and $a\cdot b=ab$; these operations extend to matrices and vectors in the usual way. Much of the initial development of max algebra was motivated by applications in scheduling and discrete event systems~\cite{BCOQ92,HOW}.  While this original motivation remains, the area
is also a fertile source of problems for specialists in combinatorics and other areas of pure mathematics.
See, in particular,~\cite{LM:05,LS:09} 
 
According to Litvinov and Maslov~\cite{LM98},
tropical mathematics (also called idempotent mathematics 
due to the idempotency law $a\oplus a=a$)
can be developed in parallel with  traditional mathematics, so that many useful 
constructions and results can be translated from 
traditional mathematics to a tropical/idempotent ``shadow'' and
back.  Applying this principle to algorithms 
gives rise to the programme of making some algorithms universal,
so that they work in traditional mathematics, tropical mathematics, and over a wider class of semirings. 

There is a well-known universal algorithm, which
 derives from Gaussian elimination without pivoting. This universal version of 
Gaussian elimination was developed by Backhouse and Carr\'{e}~\cite{BC}, see also Gondran~\cite{Gon} and Rote~\cite{Rot}.
 Based on it, Litvinov et al.~\cite{LM98,LMa,LRSS} formulated a wider concept of a universal algorithm, 
and discovered some new universal versions of Gaussian elimination for Toeplitz matrices and other special kinds of matrices. 
The semiring version of the State Reduction Algorithm found in the present paper can be seen as a new development 
in the framework of those ideas.

The present paper is also a sequel of our earlier work~\cite{BGuKMS}, 
where the Markov Chain Tree Theorem was proved over the
max algebra. To this end, we remark that the max-algebraic 
analogue of probability is known and has been studied, 
e.g., by Puhalskii~\cite{Puh01} as
idempotent probability. Our work is also related to the 
papers of Minoux~\cite{Min1, Min2}. However, the Markov 
Chain Tree Theorem established in the present paper is different from the theorem of~\cite{Min1} which establishes a 
relation between the spanning tree vector and bi-determinants of associated matrices of higher dimension. 
Also, no algorithms for computing the spanning tree vector are offered in~\cite{Min1,Min2}.

Let us mention that the proof of universal Markov Chain Tree theorem given in the present paper 
generalizes a proof that can be found in a technical report of
Fenner and Westerdale~\cite{FW}.
In our development of the universal State Reduction Algorithm we build upon the
above mentioned 
State Reduction Algorithm of~\cite{She,GTH,Son99}. The work of Sonin~\cite{Son99} appears to be particularly useful here,
since it provides most of the necessary elements of the proof.  
We recommend both the works of Fenner-Westerdale~\cite{FW} and
Sonin~\cite{Son99} to the reader as well-written explanations of the Markov Chain Tree theorem and
the State Reduction Algorithm in the setting of classical probability. 
The proofs we give here are predominantly based on combining the 
arguments of these earlier works and verifying that they generalize 
to the abstract setting of commutative semirings and semifields.

When specialized to the max algebra, the universal State Reduction Algorithm 
provides a method for computing 
the maximal weight of a spanning tree in a directed network.
Of course, the problems of minimal and maximal spanning trees in graphs, 
particularly undirected graphs, have attracted much attention \cite{Jung}.  
Recall that in the case of directed graphs, 
the best known algorithm is 
the one suggested by Edmonds~\cite{Edm} and, independently, 
Chu and Liu~\cite{CL}. This algorithm has some similarities with the universal
State Reduction Algorithm (when the latter is specialized to the max algebra), 
but we will not give 
any further details on this.

 
Let us also mention that the State Reduction algorithm can be seen
as a special case of the stochastic complements technique, see 
Meyer~\cite{Meyer}.

The rest of the paper is organized as follows. 
In Section~\ref{s:mct} we obtain the universal version of
the Markov Chain Tree Theorem. In Section~\ref{s:sr} we 
formulate the universal State Reduction Algorithm and
provide a part of its proof. Section~\ref{s:sonin} is devoted to 
the proof of a particularly technical lemma
(basically following Sonin~\cite{Son99}).


\section{Markov Chain Tree Theorem in Semirings}
\label{s:mct}
A semiring $(S, +, \cdot)$ consists of a set $S$ equipped with two (abstract)
binary operations $+$, $\cdot$.  The generalized 
addition, $+$, is  commutative and associative and has an identity element $0$.  
The generalized multiplication $\cdot$ is associative and distributes 
over $+$ on both the left and the right.  There also exists a multiplicative 
identity element $1$ and the additive identity is absorbing in the the sense 
that $a \cdot 0 = 0$ for all $a \in S$.  We shall only be concerned with commutative 
semirings, in which $\cdot$ is also commutative.  Next we list some 
well-known examples of semirings where Theorem \ref{t:mst} 
is valid.

\begin{example}
\label{ex:na}
Classical nonnegative algebra which consists of the set of all nonnegative real numbers together with the usual addition and multiplication is a commutative (but not idempotent) semiring.
\end{example}
\begin{example}
\label{ex:ma}
What we are referring to as the max algebra 
is often called the max-times algebra to distinguish it 
from other isomorphic realisations.  The max-plus algebra 
(isomorphic to max algebra via the mapping 
$x\rightarrow \exp(x)$) consists of $S=\RR \cup \{-\infty\}$
 with the operations $a+b=\max(a,b)$ and $a\cdot b=a+b$.  The min-plus 
algebra (isomorphic to max plus algebra by the mapping $x\rightarrow -x$) 
consists of $S=\RR \cup \{+\infty\}$ with the operations $a+b=\min(a,b)$ and 
$a\cdot b=a+b$.  All of these realisations are commutative idempotent semirings.
\end{example}
\begin{example}
Let $U$ be a set, and consider a Boolean algebra of subsets of $U$. This is an
idempotent semiring where $a+b=a\cup b$ and $a\cdot b=a\cap b$ for 
any two subsets $a,b\subseteq U$. In the case of finite $U$,  matrix algebra over 
$U$ was considered, e.g., by Kirkland and Pullman~\cite{KP}.
\end{example}
\begin{example}
The max-min algebra consisting 
of $S=\RR \cup \{-\infty\}\cup\{+\infty\}$ equipped with 
$a+b=\max(a, b)$ and $a\cdot b=\min(a, b)$ for all $a, b\in S$ is another commutative idempotent semiring.
\end{example}
\begin{example}
Given a semiring $S$ with idempotent addition ($a+a=a$),  
equipped with the canonical partial order $a \preceq b$ iff $a + b = b$, an 
Interval Semiring $I(S)$ (see \cite{LS00}) can be constructed as follows.  $I(S)$ consists of 
order-intervals $[a_1, a_2]$ (where $a_1 \preceq a_2$) and is equipped with the operations 
$+$ and $\cdot$ defined by 
$[a_1, a_2]+ [b_1, b_2] = [a_1+ b_1, a_2 + b_2], [a_1, a_2] \cdot [b_1, b_2] = 
[a_1\cdot b_1, a_2 \cdot b_2]$. 
\end{example}

We define addition $A+B$ and multiplication $AB$ of matrices over 
$S$ in the standard fashion.  Given a matrix 
$A \in S^{n \times n}$, the weighted directed graph 
$D(A)$ is defined in exactly the same way as for matrices with real entries.  

Let us proceed with some graph-theoretic definitions. 
By a {\em (spanning) $i$-tree} we mean a (directed) spanning tree rooted at $i$ and directed
towards $i$. A {\em functional graph} $(V, E)$ is a directed graph in which each vertex has exactly one outgoing edge.  
Such graphs are referred to as "sunflower graphs" in \cite{HOW}.  
It is easy to see that a functional graph in general contains several cycles,
which do not intersect each other.  A functional graph having only one cycle that goes through $i$ and is not a loop
(that is, not an edge of the form $(i,i)$) will be called {\em $i$-unicyclic}. 

\if{
\begin{lemma}
\label{l:equiv}
A functional graph contains an $i$-tree if and only if
it has only one cycle, and this cycle contains node $i$.
\end{lemma}
\begin{proof}
If a functional graph contains an $i$-tree then every cycle in this graph contains $i$.
However, the cycles of a functional graph do not intersect, hence this cycle is unique.

Conversely, if a functional graph contains a unique cycle, and this cycle goes through $i$, then deleting
the edge issuing from $i$ on that cycle we obtain a graph that does not contain any cycles and has a unique sink, 
hence this is an $i$-tree.
\end{proof}

\begin{corollary}
\label{c:equiv}
A functional graph contains an $i$-tree and the edge $(i,j)$ if and only if it has only one
cycle, and this cycle contains $(i,j)$.
\end{corollary}
}\fi

Let $T$ be a subgraph of $D(A)$. Define its {\em weight} $\pi(T)$ as the product 
 of the weights of the edges in $T$. 
We will use this definition only in the cases when $T$ is a directed spanning tree or a unicyclic functional 
graph. By the total weight of a set of graphs (for example, the set of all $i$-trees or all $i$-unicyclic
functional graphs) we mean the sum of the weights of all graphs in the set.

We now present a semiring version of the Markov Chain Tree Theorem. This proof
is a semiring extension of the proof in Fenner-Westerdale~\cite{FW}. See also Fre\u{\i}dlin-Wentzell~\cite{FW84}~Lemma~3.2 
and Sonin~\cite{Son99}, Lemma 6.  

We denote the set of all $i$-trees in $D(A)$ by $\mathcal{T}_i$.  The 
Rooted Spanning Tree (RST) vector $w \in S^n$ is defined by 

\begin{equation}
\label{eq:wsemi}
w_i = \sum_{T \in \mathcal{T}_i} \pi(T), i=1, \ldots, n. 
\end{equation}

In general, the set $\mathcal{T}_i$ may be empty and then $w_i=0$.
In the usual algebra and in the max algebra, $w$ is positive when $A$ is 
irreducible. 

A matrix $A \in S^{n \times n}$ is said to be {\em stochastic} 
if $a_{i1} + a_{i2} + \cdots + a_{in}=1$ 
for $1 \leq i \leq n$.

{\bf Markov Chain Tree Theorem in Semirings}

\begin{theorem}
\label{t:mst}
Let $A \in S^{n \times n}$ and 
let $w$ be defined by (\ref{eq:wsemi}).  Then for each $i=1, \ldots, n,$ we have 
\begin{equation}
\label{e:balance}
w_i\cdot\sum_{j\neq i} a_{ij}=\sum_{j\neq i} w_j a_{ji}. 
\end{equation}
If $A$ is stochastic then
\begin{equation}
\label{e:mct}
A^T \cdot w = w.
\end{equation}
\end{theorem}
\begin{proof}
To prove~\eqref{e:balance} we will argue that both parts are equal to the total weight of
all $i$-unicyclic functional digraphs, which we further denote by $\pi[i]$.  

On the one hand, every combination of an $i$-tree and an edge $(i,j)$ with $j\neq i$ results in an $i$-unicyclic 
functional digraph.  Indeed, the resulting digraph is clearly functional; moreover, every cycle in it has to contain the edge $(i,j)$,
so there is only one cycle. Hence, using the distributivity, 
the left hand side of~\eqref{e:balance} can be represented as sum of weights of
some $i$-unicyclic functional digraphs.
As each $i$-unicyclic functional digraph is uniquely determined by an $i$-tree and an 
edge $(i,j)$ where $j\neq i$,  the above mentioned sum contains all weights 
of such digraphs, with no repetitions.
Thus the left hand side of~\eqref{e:balance} is
equal to $\pi[i]$.

On the other hand, every combination of a $j$-tree and an edge $(j,i)$ with $j\neq i$ also results in an
$i$-unicyclic functional digraph (since every cycle in the 
resulting functional graph has to contain the edge $(j,i)$).
Hence, using the distributivity, 
the right hand side of~\eqref{e:balance} 
can be also represented as sum of weights of
some $i$-unicyclic functional digraphs.
If we take an $i$-unicyclic functional graph then $i$ 
may have several incoming edges, but only one of them belongs to the 
(unique) cycle. Hence there is only one $j$ such that there 
is an edge $(j,i)$ and a path from $i$ to $j$ so that a $j$-tree 
exists. Thus an $i$-unicyclic functional digraph is uniquely 
determined by a $j$-tree and an 
edge $(j,i)$ where $j\neq i$, and the right hand side of~\eqref{e:balance} is
also equal to $\pi[i]$.

Equation~\eqref{e:mct} results from adding $w_i a_{ii}$ to both sides
of~\eqref{e:balance} for each $i$, and using the stochasticity of $A$.
\end{proof}

\begin{example}\label{boo_eg}{\rm{Consider the Boolean algebra over the two-element set $U=\{\sigma_1, \sigma_2 \}.$  
Observe that the $3 \times 3$ matrix $A_1=\left[ \begin{array}{ccc} 1 & \sigma_1 & 0\\ \sigma_1 & 1 & \sigma_2 \\ 0&\sigma_2&1\end{array} \right]$ is stochastic. Referring to (\ref{eq:wsemi}), it is readily determined that the rooted spanning tree vector for $A_1$ is the zero vector. 

On the other hand, for the stochastic matrix $A_2=\left[ \begin{array}{ccc} 1 & 1 & 0\\ \sigma_1 & 1 & \sigma_2 \\ 0&\sigma_2&1\end{array} \right],$ we find that the rooted spanning  tree vector is $\left[\begin{array}{ccc} 0&\sigma_2 &\sigma_2 \end{array}\right].$ We note in passing that for the matrix $A_2,$ the techniques of \cite{KP} can be used to show that the vectors $\left[\begin{array}{ccc} 1&1 &\sigma_2 \end{array}\right]$ and $\left[\begin{array}{ccc} 0&\sigma_2 &1 \end{array}\right]$ form a basis for the left eigenspace of $A_2$ corresponding to the eigenvalue $1$.
}}
\end{example}

\section{State reduction algorithm in semifields}
\label{s:sr}
In this section, we describe an algorithm for computing the spanning tree vector $w$ 
in anti-negative semifields.  We first recall some necessary definitions.
A semiring $(S, +, \cdot)$ is called a semifield if every nonzero 
element of $S$ has a multiplicative inverse. The semirings 
in examples \ref{ex:na} and \ref{ex:ma} are
commutative semifields.

A semifield $S$ is antinegative if $a + b = 0$ implies that $a = b = 0$ for $a, b \in S$. 

Algorithm \ref{a:usr} below provides a universal version of the state reduction algorithm.
Following~\cite{LRSS} we describe this in a language derived from MATLAB. The basic arithmetic 
operations here are $a+b$, $ab$ and $inv(a):=a^{-1}$. For simplicity, we avoid making too much use of MATLAB
vectorisation here. However, we exploit the functions ``sum'' and, respectively, 
``prod'', which sum up and, respectively,
take product of all the entries of a given vector.

\begin{algorithm}
\label{a:usr}
State reduction algorithm for anti-negative semifields.
\end{algorithm}

\sssn {\bf Input:} An $n\times n$ matrix $A$ with entries $a(i,j)$ and 
at least one non-zero off-diagonal entry in each row,\\
$A$ is also used to store intermediate results of the computation process.

\msn {\em Phase 1: State Reduction}

\msn {\bf for} $i=1: n-1$\\
$s(i)=sum(a(i,i+1: n))$\\
{\bf for} $k=i+1: n$\\
{\bf for} $l=i+1: n$\\
$a(k,l)=a(k,l)
+ a(k,i)\cdot a(i,l)\cdot inv(s(i))$\\
{\bf end}\\
{\bf end}\\
{\bf end}\\

\msn {\em Phase 2: Backward Substitution}

\msn $w(n)=prod(s(1: n-1))$\\
$w(1: n-1)=0$\\
{\bf for} $i=n-1: -1: 1$\\
{\bf for} $k=i+1:n$\\
$w(i)= w(i)+w(k)\cdot a(k,i)\cdot inv(s(i))$\\
{\bf end}\\
{\bf end}


In order for the algorithm to work, it is necessary to ensure 
that the elements $s_i$ are non-zero at each step.  To this end, we assume that the 
matrix $A$ has at least 1 non-zero off-diagonal element in each row.  Formally, 
for $1 \leq i \leq n$, there exists some $j \neq i$ such that $a_{ij} \neq 0$.  A simple 
induction using the next lemma then shows that $s_i$ will 
be non-zero at each stage of Algorithm~\ref{a:usr}, Phase 1.

\begin{lemma}\label{lem:red1}
Let $A \in S^{n \times n}$ have at least one non-zero off-diagonal element in each row.  Let $s = \sum_{j=2}^{n} a_{1j}$ and define $\hat{A} \in S^{n \times n}$ as follows:
\begin{itemize}
\item[(i)] $\hat{a}_{ij} = a_{ij} + s^{-1} a_{i1}a_{1j}$ for $i, j \geq 2$;
\item[(ii)] $\hat{a}_{ij} = a_{ij}$ otherwise.
\end{itemize}
Then for $2 \leq i \leq n$, there is some $j \geq 2$, $j \neq i$ with $\hat{a}_{ij} \neq 0$.
\end{lemma}
\begin{proof} Let $ i \geq 2$ be given.  By assumption, there is some $j \neq i$ with $a_{ij} \neq 0$.  If $j \geq 2$, (i) combined with the antinegativity of $S$ implies that $\hat{a}_{ij} \neq 0$.  If not, then it follows that $a_{i1} \neq 0$ and again by assumption there is some $j$ with $a_{1j} \neq 0$.  As $S$ is antinegative, it is immediate from (i) that $\hat{a}_{ij} \neq 0$.  
\end{proof}

\begin{remark}
Phase 1 is, in fact, similar to the universal LDM decomposition described in \cite{LRSS}, with algebraic inversion operations instead of algebraic closure (Kleene star).  
\end{remark}

\medskip Algorithm \ref{a:usr} requires $\frac{n^3}{3}+O(n^2)$ 
operations of addition, $\frac{2n^3}{3}+O(n^2)$ operations of 
multiplication and $n-1$ operations of taking inverse.
The operation performed in Phase 1 can be seen as a {\bf state reduction}, where a 
selected
state of the network is suppressed, while the weights 
of the edges not using that state are modified. Recall that in the usual arithmetic
and if $A$ is stochastic, the weights of edges are transition probabilities.

For instance, on 
the first step of Phase 1 
we suppress state $1$ and obtain 
a network with weights
\begin{equation}
\label{e:rm1}
a^{(1)}_{kl}=a_{kl}+ \frac{a_{k1}a_{1l}}{s_1},\quad k,l>1.
\end{equation}
We inductively define
\begin{equation}
\label{e:rmi}
a^{(i)}_{kl}=a^{(i-1)}_{kl}+ \frac{a^{(i-1)}_{ki}a^{(i-1)}_{il}}{s_i},\quad k,l>i.
\end{equation}
for $i=1,\ldots,n-1$. 
So $A^{(i)}=a_{kl}^{(i)}$ is the matrix of the reduced network
obtained on the $i$th step of Phase 1, by forgetting the states $1,\ldots,i$.

Denote by $w^{(i)}$ the spanning tree vector of the 
$i$th reduced Markov model (with $n-i$ states). 
This vector
has components $w^{(i)}_{i+1},\ldots,w^{(i)}_n$. We will further 
use the following nontrivial 
statement, whose proof (following Sonin~\cite{Son99}) will be 
recalled below in Section~\ref{s:sonin}.

\begin{lemma}
\label{l:sonin}
For all $i<k$ we have $s_i\cdot w_k^{(i)}=w_k^{(i-1)}$.
\end{lemma}
 
Let us show (modulo this Lemma) that Algorithm~\ref{a:usr} actually works.

\begin{theorem}
\label{t:mainres}
Let $S$ be a commutative anti-negative semifield and $A\in S^{n\times n}$ 
be such that every row contains at least one nonzero
off-diagonal element. Then Algorithm~\ref{a:usr}
computes the spanning tree vector of $A$. If $A$ 
is stochastic then this vector is a left eigenvector of $A$.
\end{theorem}
\begin{proof}
We will prove this theorem by induction, analyzing Phase 2 of Algorithm~\ref{a:usr}.

To begin, we show that initializing $w(n) = s_{n-1}$ and 
performing 1 step of Phase 2, $(w(n-1), w(n))$ is the spanning 
tree vector of the reduced matrix $A^{(n-2)}$ on the $2$ states $n-1, n$.  
It is easy to check that in this case, we obtain $w(n-1) = a^{(n-2)}_{n, n-1}$. 
We also have $w(n) = s_{n-1} = a^{(n-2)}_{n-1, n}$ so that in this 
case, $(w(n-1), w(n))$ is indeed the spanning tree vector of $A^{(n-2)}$ as claimed.


For the inductive step, let us make the following assertion:
If we initialize $w(n)=s_{i+1}\cdot\ldots\cdot s_{n-1}$ instead 
of $w(n)=s_1\cdot\ldots\cdot s_{n-1}$
in the beginning of Phase 2, then the vector 
$w(i+1),\ldots, w(n)$ obtained on the $n-i-1$ step of Phase 2 is the 
spanning tree vector $w_{i+1}^{(i)},\ldots, w_n^{(i)}$
of the $i$th reduced network, with the states $1,\ldots,i$ suppressed. 

We have to show that with the above assertion, if we initialize
 $w(n)=s_i\cdot\ldots\cdot s_{n-1}$ then the vector $w(i),\ldots, w(n)$ obtained on the 
$n-i$ step of Phase 2 is the spanning tree vector
of the $i-1$ reduced network.   

Indeed, we have
\begin{equation}
w_{i+1}^{(i-1)}=s_i w_{i+1}^{(i)},\ldots, w_n^{i-1}=s_i w_n^{(i)},
\end{equation}
by Lemma~\ref{l:sonin}.  Combining this with the induction hypothesis 
and our choice of $w(n)$, we see that the 
components $w(i+1),\ldots,w(n)$ are indeed equal to 
the entries $w_{i+1}^{(i-1)}, \ldots, w_n^{(i-1)}$ of the spanning tree vector.
Next, observe that Algorithm~\ref{a:usr} computes $w(i)$ 
using $w(i+1), \ldots, w(n)$ via the balance equation:
\begin{equation}
\label{e:balance2}
s_iw(i)=\sum_{k>i} w_k^{(i-1)}a_{ki}^{(i-1)}.
\end{equation}
As $s_i$ is invertible, it now follows from Theorem~\ref{t:mst} that $w(i) = w_{i}^{(i-1)}$.  
\end{proof}

\section{Proof of Lemma~\ref{l:sonin}}
\label{s:sonin}
This proof follows closely that given in Sonin~\cite{Son99}, Section 5. 
Our main reason for including it in full is to verify that it generalizes 
to an arbitrary antinegative semifield and to give, in our view,
a different and more transparent explanation of the initial proof.

We have to show that $s_i\cdot w_k^{(i)}=w_k^{(i-1)}$ for all $k>i$.
It is enough to consider the case when $i=1$ and $k>1$. For convenience, let us assume $k=n$, 
so we are to prove that $s_1w_n^{(1)}=w_n$. Recall that here $w_n$ is the total 
weight of all $n$-trees,
$s_1=\sum_{j>1} a_{1j}$, and $w_n^{(1)}$ is the total weight of 
all $n$-trees in the reduced Markov
model where the weight of any edge $(k,l)$ for $k,l>1$ equals
\begin{equation}
\label{e:akl1}
a_{kl}^{(1)}=a_{kl}+ \frac{a_{k1}a_{1l}}{s_1}.
\end{equation}

In every tree $T = (V(T), E(T))$ that contributes to $w_n$ we can identify 
the set $D$ of nodes $i$ such that $(i, 1) \in E(T)$ (the edge 
originating at $i$ terminates at $1$).
Further, each tree contributing to $w_n$ is uniquely determined by 
1) the set $D$, 2) the forest $F$ whose (directed) trees are rooted
at the nodes of $D\cup\{n\}$, and 3) the edge starting at node $1$ and 
ending at a node of the tree rooted at $n$.

In contrast to the case of $w_n$, $w_n^{(1)}$ (using the distributivity property of $S$)
can be written as a sum of terms, where each term is determined not only by an $n$-tree on the set $\{2,\ldots,n\}$, but also by the choice of
the first or the second term in~\eqref{e:akl1}, made for each edge of the tree.  For every such term we can identify the
set of nodes $\Tilde{D}$ such that for each edge starting at one of these nodes the second term in~\eqref{e:akl1}
is chosen.  Further, each term contributing to $w_n^{(1)}$ is 
uniquely determined by 1) the set $\Tilde{D}$, 2) the forest $\Tilde{F}$ whose
trees are rooted at the nodes of $\Tilde{D}\cup\{n\}$ and 3) by the 
mapping $\tau$ from $\Tilde{D}$ to $\{2,\ldots,n\}$
(which is, in general, neither surjective nor injective).

Given a forest $F$ on the set $D\cup\{n\}$ and $k\in D\cup\{n\}$, 
we denote by $T_k(F)$ the tree rooted at $k$.

In view of the above and making use of the distributivity property 
of $S$, the equation $w_n=s_1w_n^{(1)}$ is equivalent to the following:
\begin{equation}
\label{e:major}
\begin{split}
&\sum_{D,F}\left(
\prod_{l\in D} 
a_{l1}\cdot\prod_{(i,j)\in F} a_{ij} 
\cdot \sum_{k\in T_n(F)} 
a_{1k}\right)=\\
& s_1\cdot\sum_{\Tilde{D},\Tilde{F}} 
s_1^{-|\Tilde{D}|}
\cdot\left(\prod_{l\in\Tilde{D}} a_{l1}\cdot
\prod_{(i,j)\in F} a_{ij}\cdot\sum_{\tau\colon\Tilde{D}\to\{2,\ldots,n\}}
\prod_{k\in\Tilde{D}} a_{1\tau(k)}\right).
\end{split}
\end{equation}

As the set of all pairs $(D,F)$ and the set of all pairs $(\Tilde{D},\Tilde{F})$ are identical, we are left to prove
the following identity

\begin{equation} 
\label{e:minor}
s_1^{|D|-1}\cdot \sum_{k\in T_n(F)} a_{1k}=
\sum_{\tau\colon D\to\{2,\ldots,n\}}\prod_{k\in D} a_{1\tau(k)}\quad\forall D,F .  
\end{equation}

The proof of~\eqref{e:minor} makes use of the following well-known 
combinatorial identity,
whose derivation we will briefly explain, for the reader's convenience. 
Let $T$ be an $n$-tree on $\{1,\ldots,n\}$, and let $\cT$
be the set of all $n$-trees. 
For each node $k\in\{1,\ldots,n\}$, its {\em indegree} $\indeg(k,T)$ in $T$
is defined as the number of ingoing edges. Let $x_1,\ldots,x_n$ be arbitrary
scalars from $\semr$. We will use the following version of Cayley's tree enumerator formula:

\begin{equation}
\label{e:comb2}
(x_1 +\ldots + x_n)^{n-2}\cdot x_n=\sum_{T\in\cT} x_1^{\indeg(1,T)}\cdot\ldots\cdot x_n^{\indeg(n,T)}.
\end{equation}

Recall that this formula admits a classical proof which works in any commutative semiring. 
Indeed, observe that for each term on the right hand side of~\eqref{e:comb2}, there is 
at least one variable
among $x_1,\ldots,x_{n-1}$ which does not appear, since each tree has at least one leaf. 
The same is true about the left hand side. 
of~\eqref{e:comb2}, since any monomial in the expansion of $(x_1+\ldots+ x_n)^{n-2}$ has total degree $n-2$, which is
one less than $n-1$. Due to this observation, it suffices to prove

\begin{equation}
\label{e:comb21}
(x_2 +\ldots + x_n)^{n-2}\cdot x_n=
\sum_{T\in\cT\colon 1\,\text{is a leaf}} x_2^{\indeg(2,T)}\cdot\ldots\cdot x_n^{\indeg(n,T)}.
\end{equation}

Observe that by induction (whose basis for $n=2$ is trivial) we have
\begin{equation}
\label{e:comb22}
(x_2 +\ldots + x_n)^{n-3}\cdot x_n=
\sum_{T\in\cT'} x_2^{\indeg(2,T)}\cdot\ldots\cdot x_n^{\indeg(n,T)},
\end{equation}
where $\cT'$ is the set of all (directed) $n$-trees on nodes $2,\ldots,n$.
Multiplying both parts of~\eqref{e:comb22} by $(x_2 +\ldots + x_n)$ and using the identity
\begin{equation}
\label{e:comb23}
\begin{split}
&\sum_{T\in\cT\colon 1\,\text{is a leaf}} x_2^{\indeg(2,T)}\cdot\ldots\cdot x_n^{\indeg(n,T)}=\\
&(x_2+\ldots +x_n)\cdot \sum_{T\in\cT'} x_2^{\indeg(2,T)}\cdot\ldots\cdot x_n^{\indeg(n,T)},
\end{split}
\end{equation}
which is due to the bijective correspondence between
 the trees in $\cT$ having node $1$ as a leaf and
the combinations of trees in $\cT'$ and edges issuing from node $1$, we obtain~\eqref{e:comb21} and
hence~\eqref{e:comb2}.

To apply~\eqref{e:comb2}, observe first that each mapping $\tau$ in~\eqref{e:minor} defines a mapping on $D\cup\{n\}$:
we put an edge $(u,v)$ for $u,v\in D\cup\{n\}$ if $\tau(u)$ belongs to the tree rooted at $v$. Further, 
this mapping defines a directed tree on $D\cup\{n\}$, rooted at $n$. In particular, observe that 
any cycle induced by $\tau$ would yield a cycle in the original graph
(which is a spanning tree on the nodes $2,\ldots, n$ rooted at $n$). Also, none of the nodes except for
$n$ can be a root since $\tau$ is defined for all nodes of $D$. We will refer to 
such a tree on $D\cup\{n\}$ as a $\tau$-induced tree, or just induced tree if
the mapping is not specified.

For any pair $(D,F)$ and for any $n$-tree $T$ on $D\cup\{n\}$ we can find a 
mapping $\tau\colon D\to\{2,\ldots,n\}$ which yields $T$ as a $\tau$-induced tree. 
Thus for any given pair $(D,F)$, the set of all 
possible induced trees (with all possible $\tau$), 
coincides with the set of all $n$-trees
on $D\cup\{n\}$. This set will be further denoted by $\cT_{induced}$. 

Let us set $x_l=\sum_{k\in T_l(F)} a_{1k}$ for all $l\in D\cup\{n\}$. 
Applying~\eqref{e:comb2} to the set of all reduced trees, with these $x_l$, 
a fixed pair $D,F$, and $|D|+1$ instead of $n$, we have

\begin{equation}
\label{e:minor2}
s_1^{|D|-1}\sum_{k\in T_n(F)} a_{1k}=\sum_{T\in\cT_{induced}}\left(\prod_{l\in D\cup\{n\}}\left(
\sum_{k\in T_l(F)} a_{1k}\right)^{\indeg(l,T)}\right)\quad \forall D,F.
\end{equation}

We are left to show that the right-hand sides 
of~\eqref{e:minor} and~\eqref{e:minor2} coincide.

For an induced tree $T$,
let $\tau\colon D\toT\{2,\ldots n\}$ denote the fact that $T$ is $\tau$-induced.
For each $l\in D\cup\{n\}$ let $\inset(l,T)$ denote the set of in-neighbours of $l$.
Consider the following chain of equalities, with $D$ and $F$ fixed.
\begin{equation}
\label{e:chain}
\begin{split}
& \sum_{\tau\colon D\to\{2,\ldots,n\}} 
\left(\prod_{k\in D} a_{1\tau(k)}\right) =
\sum_{T\in\cT_{induced}}\left(\sum_{\tau\colon D\toT\{2,\ldots,n\}}
\left(\prod_{k\in D} a_{1\tau(k)}\right)\right)\\
&=\sum_{T\in\cT_{induced}}\prod_{l\in D\cup\{n\}}\left(\sum_{\sigma\colon\inset(l,T)\to T_l(F)}
\prod_{s\in\inset(l,T)} a_{1,\sigma(s)}\right)\\
&=\sum_{T\in\cT_{induced}}\prod_{l\in D\cup\{n\}}\left(\sum_{k\in T_l(F)} a_{1k}\right)^{\indeg(l,T)}.
\end{split}
\end{equation}

These equalities can be explained as follows.
On the first step, we classify mappings $\tau$ according to the induced 
trees that they yield. 
On the next step, $D$ is represented as a union over all sets $\inset(l,T)$ where $l\in D\cup\{n\}$, and
we use the fact that each $\tau\colon D\toT\{2,\ldots n\}$  can be decomposed into 
a set  of some ``partial'' mappings $\sigma\colon\inset(l,T)\to T_l(F)$, 
and vice versa; every combination of such ``partial'' mappings 
gives rise to a mapping $\tau$ that yields $T$ (as a $\tau$-induced tree). 
On the last step we use the multinomial semiring identity
\begin{equation}
\label{e:binom}
\left(\sum_{k\in T_l(F)} a_{1k}\right)^{\indeg(l,T)}=\sum_{\sigma\colon\inset(l,T)\to T_l(F)}
\left(\prod_{s\in\inset(l,T)} a_{1,\sigma(s)}\right).
\end{equation}
To understand this identity observe that the left hand side of~\eqref{e:binom}
is a product of $\indeg(l,T) = |\inset(l,T)|$ identical sums of $|T_l(F)|$ terms.
By distributivity, this product can be written as a sum of monomials, where 
each monomial corresponds to a combination of choices made in each
bracket, and hence to a mapping $\sigma\colon\inset(l,T)\to T_l(F)$.
 
Finally, by~\eqref{e:chain} the right-hand sides of~\eqref{e:minor} and~\eqref{e:minor2} are equal, and this
completes the proof.

\end{document}